\documentclass[12pt,a4paper,english,reqno]{amsart}
\usepackage[a4paper,footskip=1.5em]{geometry}
\usepackage{amsmath,amssymb,amsthm,mathtools,bbm}
\usepackage[mathscr]{euscript}
\usepackage[usenames,dvipsnames]{color}
\usepackage{adjustbox,tikz,calc,graphics,babel,standalone}
\usepackage{subcaption}
\usepackage{csquotes,enumerate,verbatim}
\usepackage[final]{microtype}
\usepackage[numbers]{natbib}
\usetikzlibrary{shapes.misc,calc,intersections,patterns,decorations.pathreplacing}
\usetikzlibrary{arrows,shapes,positioning}
\usetikzlibrary{decorations.markings}
\usepackage{hyperref}
\hypersetup{colorlinks=true,linkcolor=blue,citecolor=blue,pdfpagemode=UseNone,pdfstartview={XYZ null null 1.00}}
\usepackage{cmtiup}

\tikzstyle arrowstyle=[scale=1]

\theoremstyle{plain}
\newtheorem*{theorem*}{Theorem}
\newtheorem{theorem}{Theorem}[section]
\newtheorem{lemma}[theorem]{Lemma}

\newtheorem{proposition}[theorem]{Proposition}
\newtheorem*{claim*}{Claim}

\newtheorem{conjecture}[theorem]{Conjecture}

\theoremstyle{remark}

\def\N{\mathbb{N}}

\def\C{\mathcal}

\def\SS{10^{25}}
\def\FF[#1]{2^{k-1}}

\let\originalleft\left
\let\originalright\right
\renewcommand{\left}{\mathopen{}\mathclose\bgroup\originalleft}
\renewcommand{\right}{\aftergroup\egroup\originalright}

\makeatletter
\def\imod#1{\allowbreak\mkern10mu({\operator@font mod}\,\,#1)}
\makeatother

\begin{document}

\title{Catching a fast robber on the grid}

\author{Paul Balister}
\address{Department of Mathematical Sciences, University of Memphis, Memphis TN 38152, USA}
\email{pbalistr@memphis.edu}

\author{B\'{e}la Bollob\'{a}s}
\address{Department of Pure Mathematics and Mathematical Statistics, University of Cambridge, Wilberforce Road, Cambridge CB3\thinspace0WB, UK, {\em and\/}
Department of Mathematical Sciences, University of Memphis, Memphis TN 38152, USA, {\em and\/} London Institute for Mathematical Sciences, 35a South St., Mayfair, London W1K\thinspace2XF, UK}
\email{b.bollobas@dpmms.cam.ac.uk}

\author{Bhargav Narayanan}
\address{Department of Pure Mathematics and Mathematical Statistics, University of Cambridge, Wilberforce Road, Cambridge CB3\thinspace0WB, UK}
\email{b.p.narayanan@dpmms.cam.ac.uk}

\author{Amy Shaw}
\address{Department of Mathematical Sciences, University of Memphis, Memphis TN 38152, USA}
\email{amy.shaw@memphis.edu}

\date{12 April 2015}

\begin{abstract}
We study the problem of cops and robbers on the grid where the robber is allowed to move faster than the cops. It is a well-known fact that two cops are necessary and sufficient to catch the robber on any finite grid when the robber has unit speed. Here, we prove that if the speed of the robber exceeds a sufficiently large absolute constant, then the number of cops needed to catch the robber on an $n \times n$ grid is $\exp( \Omega(\log n / \log \log n))$.
\end{abstract}

\maketitle

\section{Introduction}
The game of \emph{Cops and Robbers}, introduced almost thirty years ago independently by Nowakowski and Winkler~\citep{intro_cop1} and Quilliot~\citep{intro_cop2}, is a perfect information pursuit-evasion game played on an undirected graph $G$ as follows. There are two players, a set of cops and one robber. The game begins with the cops being placed onto vertices of their choice in $G$ and then the robber, being fully aware of the placement of the cops, positions himself at a vertex of his choosing. Afterwards, they move alternately, first the cops and then the robber along the edges of the graph $G$.  In the cops' turn, each cop may move to an adjacent vertex, or remain where he is, and similarly for the robber; also, multiple cops are allowed to occupy the same vertex. The cops win if at some time there is a cop at the same vertex as the robber; otherwise, the robber wins. The minimum number of cops for which the cops have a winning strategy, no matter how the robber plays, is called the \emph{cop number} of $G$. 

Perhaps the most well-known problem concerning the game of cops and robbers is Meyniel's conjecture which asserts that $O(\sqrt{n})$ cops are sufficient to catch the robber on any $n$-vertex graph. While Meyniel's conjecture has attracted a great deal of attention, progress towards the conjecture in its full generality has been rather slow; see~\citep{meyniel_survey} for a broad overview, and~\citep{best1, best2} for the state of the art.

In this note, we shall be concerned with a variant of the question where the robber is allowed to move faster than the cops. Let us suppose that the cops move normally as before while the robber is allowed to move at speed $R \in \N$; in other words, the robber may, on his turn, take any walk of length at most $R$ from his current position that does not pass through a vertex occupied by a cop. The definition of the cop number in this setting is analogous. This variant was originally considered by Fomin, Golovach, Kratochv{\'{\i}}l, Nisse and Suchan~\citep{fast_rob1} and following them, Frieze, Krivelevich and Loh~\citep{fast_rob2}, Mehrabian~\citep{fast_rob3}, and Alon and Mehrabian~\citep{fast_rob4} have obtained  results about how large the cop number of an $n$-vertex graph can be when the robber has a fixed speed $R > 1$.

It is natural to ask how the cop number of a given graph changes, if at all, when the speed of the robber increases from $1$ to some $R > 1$. The most natural example of a graph where this question is interesting is the $n \times n$ grid of squares where two squares of the grid are adjacent if and only if they share an edge. Let us write $f_R(n)$ for the minimum number of cops needed to catch a robber of speed $R$ on an $n \times n$ grid. Maamoun and Meyniel~\citep{fast_rob0} showed, amongst other things, that $f_1 (n) = 2$ for all $n \ge 2$. However, the flavour of the problem changes completely as soon as the robber is allowed to move faster than the cops. Nisse and Suchan~\citep{fast_grid} showed that $f_2(n) = \Omega (\sqrt{\log n})$. Our aim in this note is to prove the following extension.

\begin{theorem}\label{t:main}
There exists an $R \in \N$ and a $c_R > 0$ such that for all sufficiently large $n \in \N$, we have
\[f_R (n) \ge \exp\left( \frac{c_R \log n} {\log \log n}\right).\]
\end{theorem}

To keep the presentation simple, we shall make no attempt to optimise the speed of the robber; we prove Theorem~\ref{t:main} with $R = \SS$. 

Note that $f_R (n) \le n$ for every $R \in \N$ since $n$ cops can catch a robber of any speed on the $n \times n$ grid by lining up on the bottom edge of the grid and then marching upwards together. We suspect that this trivial upper bound is closer to the truth than Theorem~\ref{t:main}; we conjecture the following.

\begin{conjecture}
For all sufficiently large $R \in \N$, $f_R (n) = n^{1 - o(1)}$ as $n \to \infty$.
\end{conjecture}

We give a sketch of the proof of Theorem~\ref{t:main} and then the proof proper in Section~\ref{s:proof}. We also describe a modest improvement of the trivial upper bound in Section~\ref{s:upper}. We conclude with some discussion in Section~\ref{s:end}.

\section{Proof of the main result}\label{s:proof}
Our proof of Theorem~\ref{t:main} is inspired by the strategy used by Bollob\'as and Leader~\citep{angel3d} and Kutz~\citep{kutz} to resolve Conway's angel problem in three dimensions.

We fix a large positive integer $R \in \N$ which will denote the speed of the robber in what follows. We also fix two other positive integers $C, N \in \N$ such that $C$, $N$ and $R$ together satisfy $C \ge 40$, $N > 100e^{C}$ and $R > 50N$. We may, for example, take $C = 40$, $N = 10^{20}$ and $R = \SS$.

Define a sequence of grids as follows: let $\C{A}_0$ be an $N \times N$ grid and for $k\ge1$ let $\C{A}_k$ be a $(2k+1)^2 \times (2k+1)^2$ array of copies of $\C{A}_{k-1}$. 

We shall imagine that our $n \times n$ grid is tiled with copies of $\C{A}_0$, with these copies of $\C{A}_0$ themselves fitting together to form copies of $\C{A}_1$, and so on. We call each copy of $\C{A}_k$ in our grid a \emph{$k$-cell}. Our strategy for the robber will be inductive: we shall describe how the robber may run from $k$-cell to adjacent $k$-cell, the path of the robber within a $k$-cell being inductively determined, all the while avoiding $k$-cells where there are too many cops.

Let us suppose that the robber is situated on the bottom edge of a `safe' $k$-cell and wishes to get to the bottom edge of the $k$-cell above. Assume for the moment that the robber's $k$-cell is guaranteed to be `safe' for a reasonably large number of steps. 

Here then is an outline of a strategy for the robber: he plots a straight line from his current $(k-1)$-cell to a $(k-1)$-cell in the  $k$-cell above that he wishes to get to. He runs across each of the $(k-1)$-cells on the way until he reaches his destination; within each $(k-1)$-cell, his path is determined inductively. Of course, there is a problem with this strategy: along the way, a $(k-1)$-cell that the robber needs to run across might not be `safe' when he gets to it, or worse, a $(k-1)$-cell might become `unsafe' while the robber is running through it. To address these issues, the robber alters his path dynamically and detours around any $(k-1)$-cell along his planned straight line path that he finds might become `unsafe' while he is running through it. Our definition of `safety' will ensure that the robber does not have to take too many detours. It will follow, and it is here that we use the fact that a $k$-cell is a $(2k+1)^2 \times (2k+1)^2$ array of $(k-1)$-cells, that the `average speed' of the robber is large despite the fact that he has to take the occasional detour (and detours within detours, and so on). This will provide us with enough elbow room to prove what we need by induction.

We now go about making the above sketch precise. First, we define a sequence $(L_k)_{k \ge 0}$ of natural numbers by setting $L_{k} = \prod_{j=0}^{k}(2j + 1)^2$; clearly a $k$-cell is an $NL_{k} \times NL_{k}$ grid of squares. Next, we define another sequence $(T_k)_{k \ge 0}$ of natural numbers by setting $T_{0} = 1$ and $T_{k} = (2k+1)^2T_{k-1} + CT_{k-1}$ for $k \ge 1$. 

An observation that we shall use repeatedly is that each $k \ge 0$,
\[ 
T_{k} =  L_{k} \prod_{j=1}^{k}\left(1 + \frac{C}{(2j+1)^2}\right) < L_{k}\exp\left(C\sum_{j=1}^{\infty}\frac{1}{(2j+1)^2}\right) <e^{C}L_{k}.
\]

We need to define some notions of `safety'. We say that a $k$-cell is \emph{safe} at some point in time if the number of cops (at that point in time) within the $k$-cell is strictly less than $2^k$. Also, we say that a $k$-cell is \emph{safe for $t$ steps} (at some point) if the set of cops at distance at most $t$ from the $k$-cell has cardinality strictly less than $2^k$. Note that a $k$-cell safe for $t$ steps is necessarily safe for $t'$ steps for every $0 \le t' \le t$ as well.

Next, we say that a square is \emph{$k$-safe} if for each $0 \le k' \le k$, the $k'$-cell containing the square is safe for $T_{k'}$ steps; also, a square is \emph{completely $k$-safe} if it is guaranteed to be $k$-safe after a single cop move.

Notice that if a $k$-cell is safe, then it contains at most one unsafe $(k-1)$-cell. We shall require a straightforward extension of this simple observation. Let us say that two cells are \emph{separated} if they share neither an edge nor a corner. The following proposition is easily proved. 

\begin{proposition}\label{disjoint}
Let $X$ be a $k$-cell and assume that $X$ is safe for $t$ steps where $2t< N L_{k-1}$. If $P$ and $Q$ are a separated pair of $(k-1)$-cells within $X$, then either $P$ or $Q$ is safe for $t$ steps.
\end{proposition}
\begin{proof}
Simply notice that since $P$ and $Q$ are separated, the distance between them is at least $N L_{k-1}$. Consequently, the set of cops at distance at most $t$ from $P$ and the set of cops at distance at most $t$ from $Q$ are disjoint and the proposition follows.
\end{proof}

To help with the induction, we shall demarcate certain regions as `landing zones'. For $k \ge 1$, the \emph{landing zone} of a $k$-cell is the union of its bottom, top, right and left landing zones; the \emph{bottom landing zone} of a $k$-cell consists of the $3 \times 1$ sub-grid of $(k-1)$-cells at the middle of the bottom edge of the $k$-cell as shown in Figure~\ref{f:landing} and the top, right and left landing zones are analogously defined by symmetry. Also, a square is called a \emph{$k$-landing square}, if the square is contained in the landing zone of each $k'$-cell containing it for $1 \le k' \le k$.

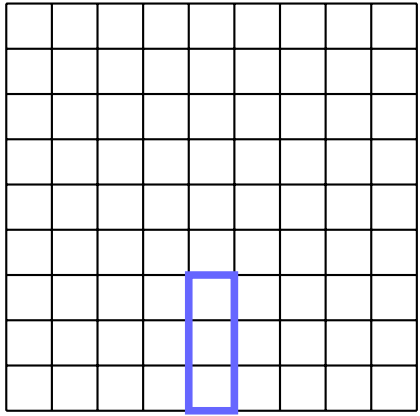
\begin{figure}
\begin{center}
\begin{tikzpicture}[scale=0.6]

\draw[step=1cm,black,thick] (0,0) grid (9,9);
\draw[blue!60!white,line width=1mm] (4,0) rectangle (5,3);

\end{tikzpicture}
\end{center}
\caption{The bottom landing zone of a $1$-cell; each square here represents a $0$-cell.}
\label{f:landing}
\end{figure}

Our proof of Theorem~\ref{t:main} hinges on the following lemma.

\begin{lemma}\label{l:main}
Let $k \ge 1$ and suppose that it is the robber's turn to move. Suppose further that the robber is positioned on a $k$-safe, $k$-landing square inside a $k$-cell $X$. If the $k$-cell $Y$ above $X$ is safe for $2T_{k}+1$ steps, then the robber has a strategy to reach, in at most $T_{k}$ steps and without getting caught, a $k$-landing square in the bottom landing zone of $Y$ which is completely $k$-safe on his arrival there.
\end{lemma}

Let us point out that there is some asymmetry in how Lemma~\ref{l:main} is stated. The lemma assumes something about the grid when the robber is about to move, and says something about the grid after a sequence of moves ending with a move made by the robber. However, note that a square is completely $k$-safe only if it is $k$-safe after a single cop move; hence, if the robber moves using the strategy given by Lemma~\ref{l:main}, then no matter how the cops move on their turn following his final move, his new location is $k$-safe (and the lemma may be applied once again).

\begin{proof}[Proof of Lemma~\ref{l:main}]
Note that Lemma~\ref{l:main} is really a collection of four different statements, one each for when the robber starts in the bottom, top, right and left landing zones of his $k$-cell $X$. Indeed, Lemma~\ref{l:main} says that under certain conditions, it is possible for the robber to safely move from the landing zone of a $k$-cell to (the landing zone of) any of its four neighbouring $k$-cells in $T_{k}$ steps.

We prove the lemma by induction on $k$. The case $k=1$ is easy to check. Assume that it is the robber's turn to move, that he is on a $1$-safe square in the landing zone of his $1$-cell $X$, and that the $1$-cell $Y$ above him is safe for $2T_{1} + 1 =19 + 2C$ steps. We need to show that he can move in at most $T_{1}$ steps to a square in the bottom landing zone of $Y$ which is completely $1$-safe on his arrival. The robber can in fact do this in one step as we now describe. 

Since the robber's square is $1$-safe, note there are no cops in his $0$-cell, say $P$. Consider a pair of separated $0$-cells, call them $Q$ and $Q'$, in the bottom landing zone of $Y$. Note that at least one of $Q$ or $Q'$, say $Q$, must be safe for two steps because if not, then since $4 < NL_{0}= N$, it follows from Proposition~\ref{disjoint} that $Y$ is not safe for two steps, contradicting our assumption that $Y$ is safe for $19 + 2C$ steps with room to spare.

Note that a $1$-cell is a $9 \times 9$ array of $0$-cells. Since a $0$-cell is an $N \times N$ grid, it is easy to see that that there are $N$ disjoint paths, each wholly contained within the union of $X$ and $Y$ and of length at most $36N$, from $P$ to any $0$-cell in $Y$. Since both $X$ and $Y$ are safe when the robber is about to move, there are at most two cops in total within $X$ and $Y$. Hence, there are at least $N-2$ paths between $P$ and $Q$ containing no cops on them. Note that the speed of the robber $R$ is greater than $36N$, and hence the robber, on his turn, can follow one of these $N-2$ paths from his square in $P$ to a square in $Q$. Note that $Q$ is safe for two steps and $Y$ is safe for $19 + 2C \ge T_{1}+1$ steps; hence, it clear that any square in $Q$ is completely $1$-safe on the robber's arrival there.

Now assume $k>1$ and that we have proved the claim for each $1 \le k' < k$. We describe the robber's strategy when he starts in the bottom landing zone of his $k$-cell. The strategy for the three other landing zones are very similar and we only highlight the very minor differences.

We shall divide the robber's journey into two parts. We first describe how the robber should travel from the bottom landing zone of his $k$-cell $X$ to the top landing zone of $X$. This journey will require at most $(2k+1)^2T_{k-1} + 3T_{k-1}$ steps. We then show that robber can dash across from the top landing zone of $X$ into the bottom landing zone of $Y$ in at most $13T_{k-1}$ steps. Hence, the total number of steps required will be bounded above by 
\[ (2k+1)^2T_{k-1} + 16T_{k-1} \le (2k+1)^2T_{k-1} + CT_{k-1} \le T_{k}\] 
as required.

In what follows, when we speak of the robber arriving at a square in a $(k-1)$-cell, it is implied that the square is a $(k-1)$-landing square.

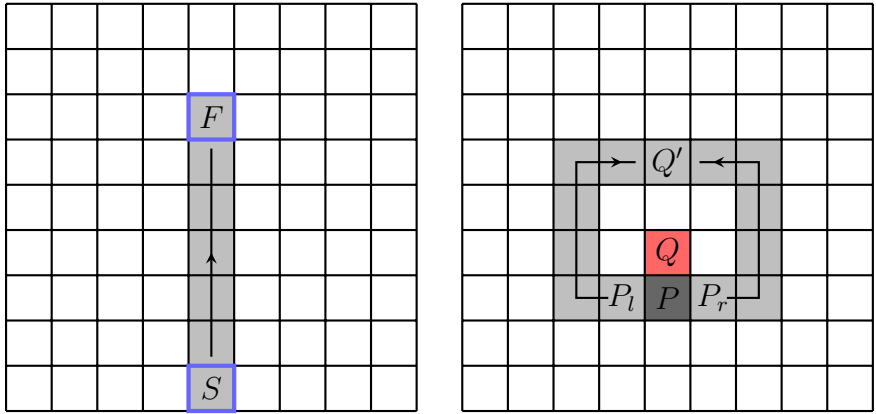
\begin{figure}
\begin{center}
\begin{tikzpicture}[scale=0.6]

\fill[black!25!white] (4,0) rectangle (5,7);
\draw[step=1cm,black,thick] (0,0) grid (9,9);
\draw[blue!60!white,ultra thick] (4,0) rectangle (5,1);
\draw[blue!60!white,ultra thick] (4,6) rectangle (5,7);

\fill[black!25!white] (12,2) rectangle (17,3);
\fill[black!25!white] (12,5) rectangle (17,6);
\fill[black!25!white] (12,3) rectangle (13,5);
\fill[black!25!white] (16,3) rectangle (17,5);
\fill[black!60!white] (14,2) rectangle (15,3);
\fill[red!60!white] (14,3) rectangle (15,4);
\draw[step=1cm,black,thick] (10,0) grid (19,9);

\node[draw=none,fill=none] at (4.5,0.5) {$S$};
\node[draw=none,fill=none] at (4.5,6.5) {$F$};

\node[draw=none,fill=none] at (14.5,2.5) {$P$};
\node[draw=none,fill=none] at (13.5,2.5) {$P_l$};
\node[draw=none,fill=none] at (15.5,2.5) {$P_r$};
\node[draw=none,fill=none] at (14.5,3.5) {$Q$};
\node[draw=none,fill=none] at (14.5,5.5) {$Q'$};

\draw[thick,postaction={decorate,decoration={markings,
		mark=at position 0.5 with {\arrow[arrowstyle]{stealth}}}}] (4.5, 1.2)--(4.5, 5.8);

\draw[thick,postaction={decorate,decoration={markings,
    mark=at position 0.94 with {\arrow[arrowstyle]{stealth}}}}] (15.8, 2.5)--(16.5, 2.5)--(16.5,5.5)--(15.2,5.5);
\draw[thick,postaction={decorate,decoration={markings,
    mark=at position 0.94 with {\arrow[arrowstyle]{stealth}}}}] (13.2, 2.5)--(12.5, 2.5)--(12.5,5.5)--(13.8,5.5);

\end{tikzpicture}
\end{center}
\caption{The planned path to the top landing zone, and the detouring strategy.}
\label{f:detour}
\end{figure}

The robber begins by plotting a straight line path from his $(k-1)$-cell, say $S$, to the nearest $(k-1)$-cell, say $F$, in the top landing zone of $X$ as shown in Figure~\ref{f:detour}. The robber's square is $k$-safe; this means that $X$ is safe for $T_{k}$ steps and that the robber's square is $(k-1)$-safe. If the $(k-1)$-cell above him is safe for $2T_{k-1} + 1$ steps, then the robber may inductively run, in at most $T_{k-1}$ steps and without getting caught, to a square in the $(k-1)$-cell above him which is completely $(k-1)$-safe on his arrival there. Following the subsequent cop turn, his square is $(k-1)$-safe. The robber may repeat this process until he gets to $F$, provided that every time the robber arrives at a $(k-1)$-cell (and the cops have subsequently moved), the $(k-1)$-cell above is safe for $2T_{k-1} + 1$ steps at that point. In this case, the robber reaches the top landing zone of $X$ in at most $(2k+1)^2T_{k-1}$ steps, and we are done.

So we may assume that at some stage of his journey, the robber is on a $(k-1)$-safe square in a $(k-1)$-cell $P$ within $X$, it is his turn to move, and that the $(k-1)$-cell $Q$ above $P$ is not safe for $2T_{k-1} + 1$ steps. We claim that the robber only has to deal with such a situation once.

Let us consider the first time such a situation arises. Clearly, the robber has taken at most $(2k+1)^2T_{k-1}$ steps from $S$; as $X$ was safe for $T_{k} = (2k+1)^2T_{k-1} + CT_{k-1}$ steps to begin with, $X$ is now safe for at least $CT_{k-1}$ steps. The robber takes a detour around $Q$ as follows. He considers the two paths around $Q$ to a $(k-1)$-cell $Q'$ located above (and separated from) $Q$ as shown in Figure~\ref{f:detour}; call these paths $\C{Z}_l$ and $\C{Z}_r$. We claim that each of the $(k-1)$-cells along one of these two paths is safe for $8T_{k-1} + 1$ steps. Indeed, all the $(k-1)$-cells on these paths with the exception of the two initial $(k-1)$-cells $P_l$ and $P_r$ are separated from $Q$. If one of these $(k-1)$-cells is not safe for $8T_{k-1} + 1$ steps, then since $Q$ is not safe for $2T_{k-1} + 1$ steps and
\[2(8T_{k-1} + 1) \le 18T_{k-1} < 18e^CL_{k-1} < NL_{k-1},\]
it follows by Proposition~\ref{disjoint} that $X$ is not safe for $8T_{k-1} + 1 \le 9T_{k-1}$ steps, contradicting the fact that $X$ is in fact, safe for $CT_{k-1}$ steps. Again, by Proposition~\ref{disjoint}, one of $P_l$ and $P_r$ is necessarily safe for $8T_{k-1} + 1$ steps since $P_l$ and $P_r$ are separated. So suppose that all the $(k-1)$-cells along $\C{Z}_l$ are safe for $8T_{k-1} + 1$ steps. Then it is easy to check that the robber may inductively run along $\C{Z}_l$, in at most $7T_{k-1}$ steps and without getting caught, from $P$ to $Q'$ so that he reaches a square in $Q'$ which is completely $(k-1)$-safe on his arrival there.

We now show that this situation arises at most once. Indeed, since $Q$ was not safe for $2T_{k-1}+1$ steps when the robber was at $P$, we know that there are at least $2^{k-1}$ cops at distance at most $2T_{k-1}+1$ from $Q$; let us mark these cops. The robber takes $7T_{k-1}$ steps to reach $Q'$ from $P$. In those $7T_{k-1}$ steps, the $2^{k-1}$ marked cops may move at most $7T_{k-1}$ steps up. However, since the distance between $Q$ and $Q'$ is $NL_{k-1} > 100e^CL_{k-1} > 100 T_{k-1}$, it is clear that these $2^{k-1}$ marked cops can never overtake the robber vertically, and hence the robber will, after this detour, always find that when he arrives at a $(k-1)$-cell, the $(k-1)$-cell above him is safe for $2T_{k-1}+1$ steps.

It is therefore clear that the robber can safely reach some $(k-1)$-cell $F$ in the top landing zone of $X$ (though, on account of his detours, not necessarily his initial choice) in at most $(2k+1)^2T_{k-1} + 3T_{k-1}$ steps. This completes the first leg of the robber's journey.

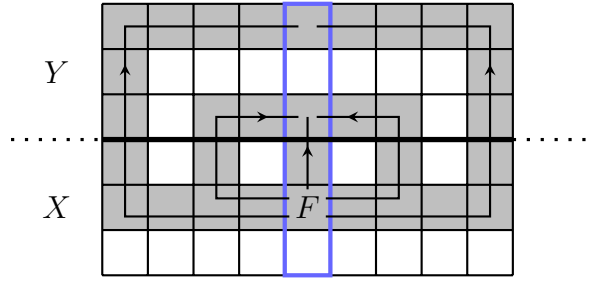
\begin{figure}
\begin{center}
\begin{tikzpicture}[scale=0.6]

\fill[black!25!white] (0,1) rectangle (9,2);
\fill[black!25!white] (2,3) rectangle (7,4);
\fill[black!25!white] (0,5) rectangle (9,6);
\fill[black!25!white] (0,1) rectangle (1,6);
\fill[black!25!white] (8,1) rectangle (9,6);
\fill[black!25!white] (2,1) rectangle (3,4);st
\fill[black!25!white] (6,1) rectangle (7,4);
\fill[black!25!white] (4,2) rectangle (5,3);

\draw[step=1cm,black,thick] (0,0) grid (9,6);
\draw[blue!60!white,ultra thick] (4,0) rectangle (5,3);
\draw[blue!60!white,ultra thick] (4,3) rectangle (5,6);
\draw[black,line width=2pt](0,3)--(9,3);
\draw[loosely dotted, very thick] (-2, 3) -- (0,3);
\draw[loosely dotted, very thick] (9, 3) -- (11,3);

\draw[thick,postaction={decorate,decoration={markings,
    mark=at position 0.6 with {\arrow[arrowstyle]{stealth}}}}] (4.5, 1.9)--(4.5, 3.5);
\draw[thick,postaction={decorate,decoration={markings,
    mark=at position 0.872 with {\arrow[arrowstyle]{stealth}}}}] (4.1, 1.7)--(2.5, 1.7)--(2.5,3.5)--(4.3,3.5);
\draw[thick,postaction={decorate,decoration={markings,
    mark=at position 0.872 with {\arrow[arrowstyle]{stealth}}}}] (4.9, 1.7)--(6.5, 1.7)--(6.5,3.5)--(4.7,3.5);

\draw[thick,postaction={decorate,decoration={markings,
    mark=at position 0.6 with {\arrow[arrowstyle]{stealth}}}}] (4.1, 1.3)--(0.5, 1.3)--(0.5,5.5)--(4.3,5.5);
\draw[thick,postaction={decorate,decoration={markings,
    mark=at position 0.6 with {\arrow[arrowstyle]{stealth}}}}] (4.9, 1.3)--(8.5, 1.3)--(8.5,5.5)--(4.7,5.5);
\node[draw=none,fill=none] at (4.5,1.5) {$F$};

\node[draw=none,fill=none] at (-1,1.5) {$X$};
\node[draw=none,fill=none] at (-1,4.5) {$Y$};
\end{tikzpicture}
\end{center}
\caption{The final stretch from $F$ to the bottom landing zone of $Y$.}
\label{f:path}
\end{figure}

Let us now pause and survey the robber's situation after the cops have moved. He is now on a $(k-1)$-safe, $(k-1)$-landing square in a $(k-1)$-cell $F$ in the top landing zone of his $k$-cell $X$. Also, $X$ is now safe for at least $(C-3)T_{k-1}$ steps and $Y$, the $k$-cell above $X$, is safe for at least $T_{k} + (C-3)T_{k-1} + 1$ steps. 

We now show that the robber can safely reach, in at most $13T_{k-1}$ steps, a square in the bottom landing zone of $Y$ which is completely $(k-1)$-safe when the robber arrives there; that this square is also completely $k$-safe follows from the fact that $Y$ is safe for at least $T_{k} + (C-3)T_{k-1} + 1$ steps before the robber starts the second leg of his journey.

It is possible to consider a set of five paths, as shown in Figure~\ref{f:path}, from $F$ to the $(k-1)$-cells in the bottom landing zone of $Y$, and show using Proposition~\ref{disjoint}, the fact that $X$ is safe for $(C-3)T_{k-1}$ steps, and the fact that $Y$ safe for $T_{k} + (C-3)T_{k-1} + 1$ steps, that all the $(k-1)$-cells along one of these five paths are all safe for $14T_{k-1} + 1$ steps. Since each of these five paths is composed of at most thirteen $(k-1)$-cells, it is clear that the robber can then complete his journey by following one of these paths in at most $13T_{k-1}$ steps. This completes the second leg of the robber's journey.

Clearly, this also shows how the robber may proceed if he is initially located in the top landing zone of $X$. 
A similar strategy to the what has just been described (see Figure~\ref{f:left}) can be easily shown to work when the robber starts on either the right or the left landing zone of $X$; the robbers path becomes slightly longer than before if he needs to make a detour as he is `turning', but our choices of $C$ and $N$ are large enough to ensure that the detouring strategy works with room to spare.
\end{proof}

Armed with Lemma~\ref{l:main}, it is a simple exercise to deduce Theorem~\ref{t:main}.

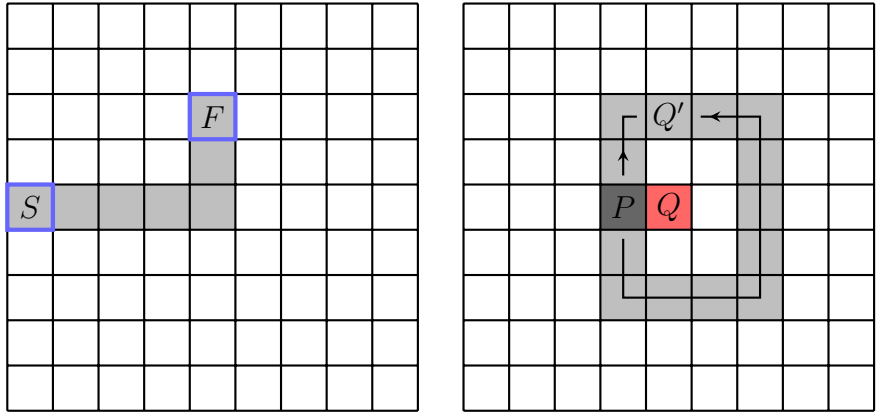
\begin{figure}
\begin{center}
\begin{tikzpicture}[scale=0.6]

\fill[black!25!white] (0,4) rectangle (5,5);
\fill[black!25!white] (4,5) rectangle (5,7);
\draw[step=1cm,black,thick] (0,0) grid (9,9);
\draw[blue!60!white,ultra thick] (0,4) rectangle (1,5);
\draw[blue!60!white,ultra thick] (4,6) rectangle (5,7);

\fill[black!25!white] (13,2) rectangle (17,3);
\fill[black!25!white] (14,6) rectangle (17,7);
\fill[black!25!white] (13,3) rectangle (14,4);
\fill[black!25!white] (16,3) rectangle (17,6);
\fill[black!25!white] (13,5) rectangle (14,7);
\fill[black!60!white] (13,4) rectangle (14,5);
\fill[red!60!white] (14,4) rectangle (15,5);
\draw[step=1cm,black,thick] (10,0) grid (19,9);

\node[draw=none,fill=none] at (0.5,4.5) {$S$};
\node[draw=none,fill=none] at (4.5,6.5) {$F$};

\node[draw=none,fill=none] at (14.5,6.5) {$Q'$};
\node[draw=none,fill=none] at (14.5,4.5) {$Q$};
\node[draw=none,fill=none] at (13.5,4.5) {$P$};

\draw[thick,postaction={decorate,decoration={markings,
    mark=at position 0.98 with {\arrow[arrowstyle]{stealth}}}}] (13.5,3.8)--(13.5, 2.5)--(16.5, 2.5)--(16.5,6.5)--(15.2,6.5);
\draw[thick,postaction={decorate,decoration={markings,
    mark=at position 0.35 with {\arrow[arrowstyle]{stealth}}}}] (13.5,5.2)--(13.5, 6.5)--(13.8,6.5);
    
\end{tikzpicture}
\end{center}
\caption{Getting to the top landing zone from the left landing zone.}
\label{f:left}
\end{figure}

\begin{proof}[Proof of Theorem~\ref{t:main}]
We show that if $n \ge 2NL_{k}$ for some $k \ge 1$, then $f_R (n) \ge 2^k$; since $L_{k} = \prod_{j = 0}^k(2j+1)^2 = \exp(O(k\log k))$, this implies the result.

Since $n \ge 2NL_{k}$, we may fix a $2 \times 2$ array of $k$-cells in the grid. If the number of cops on the grid is strictly less than $2^k$, each of these $k$-cells is guaranteed to be safe forever. After the cops have placed themselves on the grid, the robber positions himself on a $k$-safe, $k$-landing square in one of these four $k$-cells; that the robber can actually find such a square is easily checked by Proposition~\ref{disjoint}. The robber now wins by repeatedly using Lemma~\ref{l:main} to run around this $2 \times 2$ array in a clockwise loop forever.
\end{proof}

\section{Upper bounds}\label{s:upper}
We remarked earlier that $f_R(n) \le n$ for all $n \in \N$. Here, we sketch how this trivial bound may be improved slightly.

\begin{proposition}\label{p:upper}
For each $R \in \N$, we have
\[
f_R(n) \le  n \left( \frac{2R-2}{2R-1} \right) + O(1)\]
for all sufficiently large $n \in \N$.
\end{proposition}
\begin{proof}
We describe a winning strategy for $2L+1$ cops where $L$ is an integer such that 
\[L \ge n\left(\frac{R-1}{2R-1}\right) + 2.
\]

The cops initially arrange themselves in a line on the top row of the grid as shown in Figure~\ref{f:wedge}. We write $(x_r,y_r)$ and $(x_c,y_c)$ respectively for the positions of the robber and the central cop in the formation, with the convention that the bottom left square of the grid is $(1,1)$.

The cops all move together as follows. When the cops have to move, they all step down if either
\begin{enumerate}
\item $|x_c - x_r| \le R$, or
\item $x_c = L + 1$ and $x_r < x_c - R$, or
\item $x_c = n - L - 1$ and $x_r > x_c + R$.
\end{enumerate}
If none of those conditions are satisfied, then the cops all step to the left if $x_c > x_r$ and to the right otherwise.

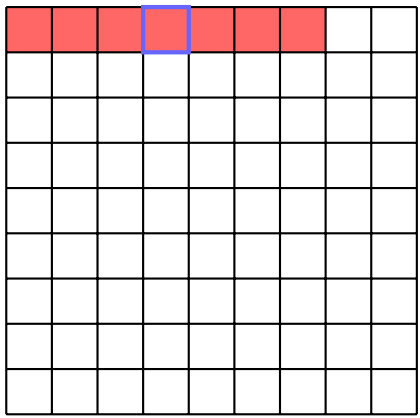
\begin{figure}
\begin{center}
\begin{tikzpicture}[scale=0.6]

\fill[red!60!white] (0,8) rectangle (7,9);
\draw[step=1cm,black,thick] (0,0) grid (9,9);

\draw[blue!60!white,ultra thick] (3,8) rectangle (4,9);

\end{tikzpicture}
\end{center}
\caption{The initial formation.}
\label{f:wedge}
\end{figure}

Observe that since the robber may move at most $R$ steps in his turn, if the quantity $x_c - x_r$ changes sign in some interval, then we must have $|x_c - x_r| \le R$ after one of the robber's turns in that interval, at which point all the cops step down.

First, we claim that cops eventually leave the top row. Indeed, we may assume that $x_r > x_c + R$ after the robber's first turn, for if not the cops step down on their first turn. Now, the cops begin moving to the right one step at a time, as dictated by their strategy. By our earlier observation, we may assume that we have $x_r > x_c$ until the line of cops reaches the right edge of the grid. However, when the cops reach the right edge, we have $x_c = n - L - 1$ and $x_r > x_c$, so the cops step down on their next turn.

Notice that robber has to be in a column occupied by one of the cops when the cops step down. Since the cops start on the top row, we may assume that $y_r < y_c$ after the cops step down for the first time, for the robber would have been caught otherwise. 

We claim that this inequality, namely $y_r < y_c$, is now maintained by the strategy until the robber is caught. To see this, suppose otherwise and let $(x^1_r, y)$ and $(x^1_c, y)$ denote the positions of the robber and the central cop at the first instance at which we have $y_r \ge y_c$; here, $(x^1_r, y)$ may be an intermediate position in the robber's trajectory, i.e., a position occupied by the robber at some point \emph{during} his turn but perhaps not the position occupied by the robber at the end of his turn. Clearly, we must either have $x^1_r > x^1_c + L$ or $x^1_r < x^1_c - L$; without loss of generality, let us assume the former.

Let $(x^0_c, y)$ be the position of the central cop at the time when the cops first step down to the $y$th row, and let $(x^0_r, y')$ be the robber's position at that time. Consider the moves made by the cops during the robber's journey from $(x^0_r, y')$ to $(x^1_r, y)$. There are no steps down in this interval as this would contradict the maximality of $y$. Since $x^1_r > x^1_c$, we must have had $x_r > x_c$ the entire time, so each move in this interval is a step to the right. In particular, we must have had $x_c  < n-L-1$ at all times in this interval, so we conclude that $x^0_c \le x^0_r \le x^0_c + R$.

The robber needs at least $(x^1_r - x^0_r)/R$ steps to get from $(x^0_r, y')$ to $(x^1_r, y)$. In this period, the cops move at least $(x^1_r - x^0_r)/R - 1$ times to the right. Therefore, 
\[ x^1_c - x^0_c \ge \frac{x^1_r - x^0_r}{R} - 1 > \frac{(x^1_c + L) -  (x^0_c + R)}{R} - 1,
\]
which implies that
\[ \frac{L}{R} < (x^1_c - x^0_c)\left(\frac{R-1}{R} \right)+ 2  < (n - 2L)\left(\frac{R-1}{R}\right) + 2.\]
This is equivalent to
\[ 
L\left(\frac{2R-1}{R}\right) < n \left(\frac{R-1}{R}\right) + 2;
\]
this contradicts the fact that $L \ge n(R-1)/(2R-1) + 2$.

To finish the proof, we may argue (as we did for the top row) that the cops cannot be stuck in any row indefinitely, and they therefore catch the robber when they step down into the bottom row.
\end{proof}

The argument used to prove Proposition~\ref{p:upper} is by no means close to best-possible. For instance, by considering a similar strategy to the one in the proof where the cops begin by arranging themselves in a downward-pointing wedge instead of a straight line, one can show that
\[
f_R(n) \le  n \left( \frac{R-1}{R+1} \right) + O(1).
\]
However, our argument for proving this bound is somewhat tedious, so we omit the proof and settle for the slightly weaker bound in Proposition~\ref{p:upper} since the purpose of the proposition is merely to demonstrate that the trivial bound of $f_R(n) \le n$  is not tight.

\section{Conclusion}\label{s:end}
It seems exceedingly unlikely that Theorem~\ref{t:main} is close to the truth; as we conjectured earlier, it should be the case that there exists an $R \in \N$ for which $f_R (n) = n^{1 - o(1)}$ as $n \to \infty$. 

Our proof of Theorem~\ref{t:main} is built on ideas used to solve Conway's angel problem in three dimensions. We conclude by mentioning that it is not inconceivable that one can, by suitably adapting one of the solutions (see~\citep{angel2d-1, angel2d-2, angel2d-3}) to Conway's problem in two dimensions, prove the existence of an $R \in \N$ and a $c_R > 0$ such that $f_R(n) \ge n^{c_R}$ for all sufficiently large $n \in \N$.

\section*{Acknowledgements}
The first and second authors were partially supported by NSF grant DMS-1301614 and the second author also wishes to acknowledge support from EU MULTIPLEX grant 317532.

\bibliographystyle{amsplain}
\bibliography{fast_robber}

\providecommand{\bysame}{\leavevmode\hbox to3em{\hrulefill}\thinspace}
\providecommand{\MR}{\relax\ifhmode\unskip\space\fi MR }
\providecommand{\MRhref}[2]{%
  \href{http://www.ams.org/mathscinet-getitem?mr=#1}{#2}
}
\providecommand{\href}[2]{#2}
\begin{thebibliography}{10}

\bibitem{fast_rob4}
N.~Alon and A.~Mehrabian, \emph{On a generalization of {M}eyniel's conjecture
  on the {C}ops and {R}obbers game}, Electron. J. Combin. \textbf{18} (2011),
  Paper 19.

\bibitem{meyniel_survey}
William Baird and Anthony Bonato, \emph{Meyniel's conjecture on the cop number:
  a survey}, J. Comb. \textbf{3} (2012), 225--238.

\bibitem{angel3d}
B.~Bollob{\'a}s and I.~Leader, \emph{The angel and the devil in three
  dimensions}, J. Combin. Theory Ser. A \textbf{113} (2006), 176--184.

\bibitem{angel2d-2}
B.~H. Bowditch, \emph{The angel game in the plane}, Combin. Probab. Comput.
  \textbf{16} (2007), 345--362.

\bibitem{fast_rob1}
F.~V. Fomin, P.~A. Golovach, J.~Kratochv{\'{\i}}l, N.~Nisse, and K.~Suchan,
  \emph{Pursuing a fast robber on a graph}, Theoret. Comput. Sci. \textbf{411}
  (2010), 1167--1181.

\bibitem{fast_rob2}
A.~Frieze, M.~Krivelevich, and P.~Loh, \emph{Variations on cops and robbers},
  J. Graph Theory \textbf{69} (2012), 383--402.

\bibitem{angel2d-3}
P.~Gacs, \emph{The angel wins}, Preprint, arXiv:0706.2817.

\bibitem{kutz}
M.~Kutz, \emph{Conway's angel in three dimensions}, Theoret. Comput. Sci.
  \textbf{349} (2005), 443--451.

\bibitem{best1}
L.~Lu and X.~Peng, \emph{On {M}eyniel's conjecture of the cop number}, J. Graph
  Theory \textbf{71} (2012), 192--205.

\bibitem{fast_rob0}
M.~Maamoun and H.~Meyniel, \emph{On a game of policemen and robber}, Discrete
  Appl. Math. \textbf{17} (1987), 307--309.

\bibitem{angel2d-1}
A.~M{\'a}th{\'e}, \emph{The angel of power 2 wins}, Combin. Probab. Comput.
  \textbf{16} (2007), 363--374.

\bibitem{fast_rob3}
A.~Mehrabian, \emph{Lower bounds for the cop number when the robber is fast},
  Combin. Probab. Comput. \textbf{20} (2011), 617--621.

\bibitem{fast_grid}
N.~Nisse and K.~Suchan, \emph{Fast robber in planar graphs}, Graph-Theoretic
  Concepts in Computer Science, Lecture Notes in Computer Science, vol. 5344,
  Springer Berlin Heidelberg, 2008, pp.~312--323.

\bibitem{intro_cop1}
R.~Nowakowski and P.~Winkler, \emph{Vertex-to-vertex pursuit in a graph},
  Discrete Math. \textbf{43} (1983), 235--239.

\bibitem{intro_cop2}
A~Quilliot, \emph{Jeux et pointes fixes sur les graphes}, Ph.D. thesis,
  Universit{\'e} de Paris VI, 1978.

\bibitem{best2}
A.~Scott and B.~Sudakov, \emph{A bound for the cops and robbers problem}, SIAM
  J. Discrete Math. \textbf{25} (2011), 1438--1442.

\end{thebibliography}

\end{document}